\documentclass[11pt]{article}

\usepackage{amsfonts, amsmath, amsthm, amssymb, verbatim}

\newcommand{\Z}{\mathbb{Z}}
\newcommand{\C}{\mathbb{C}}
\newcommand{\N}{\mathbb{N}}
\newcommand{\R}{\mathbb{R}}
\renewcommand{\i}{\mathfrak{i}}
\renewcommand{\Re}{\mathfrak{Re}}

\newcommand{\Res}{\mathrm{Res}}
\renewcommand{\d}{\mathrm{d}}

\renewcommand{\det}{\mathrm{det}}
\newcommand{\E}{\mathbf{E}}

\renewcommand{\deg}{\mathrm{deg\,}}
\newcommand{\adj}{\mathrm{adj}}

\theoremstyle{plain}
\newtheorem{lemma}{Lemma}
\newtheorem{theorem}{Theorem}
\newtheorem{proposition}{Proposition}
\newtheorem{corollary}{Corollary}

\theoremstyle{remark}

\begin{document}

\title{High-frequency sampling of multivariate CARMA processes}

\author{P\'eter Kevei\thanks{This research was funded by a postdoctoral fellowship
of the Alexander von Humboldt Foundation.} \\
Center for Mathematical Sciences, Technische Universit\"at M\"unchen \\
Boltzmannstra{\ss }e 3, 85748 Garching, Germany \\
\texttt{peter.kevei@tum.de}}

\date{}

\maketitle

\begin{abstract}
High-frequency sampled multivariate continuous time autoregressive moving average
processes are investigated. We obtain asymptotic expansion for the spectral density
of the sampled MCARMA process $(Y_{n\Delta})_{n \in \Z}$ as $\Delta \downarrow 0$, where
$(Y_t)_{t \in \R}$ is an MCARMA process. We show that the properly filtered process
is a vector moving average process, and determine the asymptotic moving average
representation of it, thus  generalizing the results by Brockwell et
al.~\cite{BFK1, BFK2} in the univariate case to the multivariate model.
The determination of the moving average representation 
of the filtered process, important for the analysis of 
high-frequency data, is difficult for any fixed positive $\Delta$. 
However, the results established here 
provide a useful and insightful approximation when $\Delta$ is very small.
\\
\noindent \textit{Keywords:} Multivariate continuous time autoregressive moving average 
(CARMA) process;
spectral density; high-frequency sampling; discretely sampled process. \\
\noindent \textit{AMS 2000 Subject Classification:} Primary 60G12, 62M10, 62M15; \\
\noindent \phantom{\textit{AMS 2000 Subject Classification:\,}} Secondary 60G10, 60G25,
60G35.
\end{abstract}

\section{Introduction}

The main object of this paper is the multivariate continuous time autoregressive 
moving average process 
(MCARMA) in $d$ dimension, which we define as follows. Let
\[
\begin{split}
P(z) & = I_d z^p + A_1 z^{p-1} + \ldots + A_p, \\
Q(z) & = B_0 z^q + B_1 z^{q-1} + \ldots + B_q
\end{split}
\]
be the autoregressive and moving average polynomial respectively, $A_i, B_j \in M_d$,
and $I_d$ is the $d$-dimensional identity matrix.
The set of $m \times n$ real matrices is denoted by $M_{m,n}$, and $M_n$ for $m=n$.
The driving process is a two sided $d$-dimensional L\'evy process $(L_t)_t \in \R$, that 
is
\[
L_t =
\begin{cases}
L_1(t), & t \geq 0, \\
-L_2(-t-), & t < 0,
\end{cases}
\]
where $L_1(t), L_2(t)$, $t \geq 0$, are independent identically distributed (one-sided)
$d$-dimensional L\'evy processes, such that $\E L_1(1) = 0$, and 
$\E \| L_1(1) \|^2 < \infty$,
with $\| \cdot \|$ being the usual Euclidean norm.
The covariance matrix of $L_1(1)$ is $\Sigma_L$. 
For definition and properties of L\'evy processes we refer to Bertoin \cite{Bertoin}.

The continuous-time analog of discrete time ARMA equations is the differential equation
\[
P(D) Y_t = Q(D) D L_t, 
\]
with $D$ being the differential operator with respect to $t$. Since L\'evy processes are 
not differentiable, this is meant as the following state space representation.

The $d$-dimensional stochastic process $Y$ is an MCARMA process with autoregressive
and moving average polynomial $P$ and $Q$ respectively, if
\begin{equation} \label{eq:def-MCARMA}
\begin{split}
\d G(t) & = \mathcal{A} G(t) \d t + \mathcal{B} \d L_t, \\
Y_t & = \mathcal{C} G(t), \quad t \in \R,
\end{split}
\end{equation}
where
\[
\mathcal{A} =
\begin{pmatrix}
0 & I_d & 0 & \ldots & 0 \\
0 & 0 & I_d & \ldots & 0 \\
\vdots & \vdots & \vdots & \ldots & \vdots \\
0 & 0 & 0 & \ldots & I_d \\
- A_p & - A_{p-1} & - A_{p-2} & \ldots & -A_1
\end{pmatrix} \in M_{pd},
\quad
\mathcal{C} = ( I_d, 0, 0, \ldots , 0 ) \in M_{d,pd}, 
\]
and
\[
\mathcal{B} = ( \beta_1^\top, \beta_2^\top, \ldots, \beta_p^\top )^\top \in M_{pd,d}, \
\beta_{p-j}=
\begin{cases}
- \sum_{i=1}^{p-j-1} A_i \beta_{p-j-i} - B_{q-j}, & 0 \leq j \leq q, \\
0, & j > q.
\end{cases}
\]
Let $\lambda_1, \ldots, \lambda_{pd}$ denote
the eigenvalues of $\mathcal{A}$, which is the same as the of zeros of $\det P(z)$, see 
Lemma 3.8 in Marquardt and Stelzer \cite{MS}.
It is well-known (see Brockwell \cite{Brock00}) that in one dimension a stationary causal
solution exists if and only if the zeros of $\det P(z)$ have negative real parts. 
Under this condition strictly stationary causal solution of the MCARMA equation
(\ref{eq:def-MCARMA}) exists, see Schlemm and Stelzer \cite[Proposition 3.2]{SS}.
Therefore throughout the paper we assume that
\[
\text{the zeros of } \det P(z) \text{ have negative real part.} 
\]
Under this assumption the process $Y$ can be represented as a continuous time moving 
average process (\cite{MS} (3.38)--(3.39))
\begin{equation} \label{eq:def-Y}
Y_t = \int_{-\infty}^t g(t-s) \d L_s, \quad t \in \R,
\end{equation}
where
\[
g(t) = \frac{1}{2 \pi} \int_{-\infty}^\infty  e^{\i t x} P(\i x)^{-1} Q(\i x) \d x.
\]
By Lemma 3.24 in \cite{MS}, the assumptions on the eigenvalues of $\mathcal{A}$ implies
that $g$ vanishes on the negative half-line, that is our process is causal.
In this paper we use representation (\ref{eq:def-Y}). Since we are only interested in
second-order properties of the process $Y$, the integral in (\ref{eq:def-Y}) is
understood in the $L^2$-sense. For the same reason our results remain valid in a more 
general setup, when the process $(L_t)_{t \in \R}$ has stationary orthogonal increments.

These processes were introduced by Marquardt and Stelzer \cite{MS} in 2007 as a 
continuous
time analog of vector-valued ARMA processes, and multivariate extension of continuous time
ARMA models. The state space representation of these models were investigated by Schlemm 
and Stelzer \cite{SS}, in particular the definition given above is from \cite{SS}.
Gaussian CARMA processes date back to Doob \cite{Doob} in 1944, while more general 
L\'evy-driven CARMA models were introduced by Brockwell \cite{Brock01} in 2001. These 
models are important tools for  stochastic modeling and have a wide range of 
applications, e.g.~in financial mathematics, electricity markets, and in turbulence. 
Andresen et al.~\cite{Benth} used Gaussian CARMA processes to model the short and 
forward 
interest rate. It is worth to mention that the Vasicek model corresponds to the 
Ornstein--Uhlenbeck process, which is the CARMA(1,0) process. Todorov and Tauchen 
\cite{TodTau} applied
L\'evy-driven CARMA processes to model stochastic volatility in finance.
For a review on applications of L\'evy-driven time series models in finance we 
refer to Brockwell and Lindner \cite{BL2}.
For spot prices in electricity markets Benth et al.~\cite{BKMV} proposed a model with 
seasonality, consisting of a deterministic seasonality, a CARMA factor driven by a 
non-Gaussian stable L\'evy process, and a non-stationary longterm factor given by a 
L\'evy 
process. Brockwell et al.~\cite{BFK2} applied CARMA processes to model high-frequency 
sampled turbulence data.

Investigations of MCARMA models have become more active in the recent years. Ergodic and 
mixing properties of equidistantly sampled MCARMA processes were studied by Schlemm and 
Stelzer \cite{SS}. Fasen \cite{Fasen14} investigated asymptotic properties of 
high-frequency sampled models, and gave parameter estimation, while Fasen and Kimmig 
\cite{FasenKimmig} studied information criteria for MCARMA models. Recovery of the 
driving L\'evy process was treated by Brockwell and Schlemm \cite{BrockSchlemm}
and by Ferrazzano and Fuchs \cite{FerrFuchs}.

In the present paper we consider high-frequency sampling of an MCARMA process and
investigate the characteristics of the resulting process $Y_n^\Delta = Y_{n\Delta}$. As 
$\Delta \downarrow 0$ we obtain the asymptotic moving average representation of the
filtered process, thus extending the results in Brockwell et al.~\cite{BFK1}, and partly 
in \cite{BFK2} to the multivariate setup. Our results slightly generalize the
moving average decomposition of Schlemm and Stelzer \cite[Theorem 4.2]{SS}, and 
gives further information on the decomposition as $\Delta \downarrow 0$.
The determination of the moving average representation 
of the filtered process, important for the analysis of 
high-frequency data, is extremely difficult for any fixed
positive $\Delta$ even in the one-dimensional case. However, the results established here 
provide a useful and insightful approximation when $\Delta$ is very small.

In the next section we give two representations of the spectral density of the sampled 
process. The first one in Proposition \ref{prop:spect1} is a Taylor expansion in 
$\Delta$, 
which we use to prove the asymptotic moving average representation in Theorem 
\ref{th:marepr}. The second one in Proposition \ref{prop:spect2} and 
\ref{prop:spect2-gen} allows us to show that 
the filtered process is a moving average process. Section \ref{sect:ma-repr} contains the 
main result, the moving average representation of the filtered process. In Section 
\ref{sect:example} the simplest non-univariate example is spelled out in detail, showing 
the difficulties to obtain higher-order approximations as is possible in the 
one-dimensional case. 
All the proofs are gathered in the last section.

\section{Spectral density of the sampled process} \label{sect:density}

It is well-known that MCARMA processes have spectral density function, which is
\begin{equation} \label{eq:spect-densY}
f_Y(\lambda) = \frac{1}{2 \pi} P(\i \lambda)^{-1} Q(\i \lambda) \Sigma_L
Q(\i \lambda)^* (P(\i \lambda)^{-1})^*, \quad \lambda \in \R,
\end{equation}
where $A^*$ is the Hermite-transpose of the complex matrix $A$; see (3.43) in \cite{MS}.
Put
\begin{equation} \label{eq:Rdef}
R(z) = P(z )^{-1} Q(z) \Sigma_L Q(-z)^\top (P(-z)^{-1})^\top, \ z \in \C.
\end{equation}
Note that the components of the matrix $P(z )^{-1} Q(z)$ are meromorphic functions,
which can have poles only at the zeros of $\det P(z)$. Due to our assumptions the poles 
have negative real part.

We frequently use the following simple facts about residues.
If $h(z)$ is a meromorphic function and $\rho$ is a closed curve, which encircles the
poles of $h$, then its residue at infinity is defined as
\[
\Res (h(z), \infty) = - \frac{1}{2 \pi \i} \int_{\rho} h(z) \d z,
\]
moreover, it can be computed as
\[
\Res (h(z), \infty) = - \Res (z^{-2} h(z^{-1}) , 0).
\]
For a matrix $M(z)$ with rational function
entries let $\deg M(z)_{i,j}$ denote the degree of the numerator minus the degree of the 
denominator in the $(i,j)$-th element of $M$. Finally, for a square matrix $A$ its 
adjugate is the transpose of its cofactor matrix, that is $\adj A = C^\top$, where
$C_{i,j}= (-1)^{i+j} A_{i,j}$, and $A_{i,j}$ is the 
determinant of the matrix that results from deleting row $i$ and column $j$ of $A$.
When $A$ is non-singular then $A^{-1} = \adj A / \det A$.

From now on, let us fix a closed curve $\rho \subset (-\infty, 0) \times \i \R$ 
in the left half complex plane, which contains the poles of $R(z)$ (the zeros of
$\det P(z)$). For any non-negative integer $k$ introduce the notation
\begin{equation} \label{eq:defTheta}
\int_{\rho \cup -\rho} z^{k} R(z) \d z  =
- 2 \pi \i \, \Res \! \left( z^{k} R(z), \infty \right) =: 2 \pi \i \Theta_k.
\end{equation}
Since $P^{-1}(z) = \frac{\adj P(z)}{\det P(z)}$, we have
\[
R(z) = \frac{1}{\det P(z) \det P(-z)} \adj P(z) Q(z) \Sigma_L (\adj P(-z) Q(-z))^\top.
\]
Note that $\deg (\adj P(z)) \leq  (d-1)p$ and $\deg Q(z) \leq q$, thus we may write
\begin{equation} \label{eq:S-pol}
(\adj P(z) ) Q(z) = \sum_{j=0}^{(d-1)p +q} S_j z^j,
\end{equation}
where $S_j \in M_d$, $j =0,1, \ldots, (d-1)p +q$. Since the polynomials on the main 
diagonal of $\adj P(z)$ are of degree 
$(d-1) p$, otherwise the degrees are strictly less than $(d-1)p$, we obtain that 
$S_{(d-1)p +q} = B_0$. With this notation
\begin{equation} \label{eq:R-pol}
S(z) = 
\adj P(z) Q(z) \Sigma_L (\adj P(-z) Q(-z))^\top
= \sum_{j=0}^{2[(d-1)p+q]} \widetilde S_j z^j,
\end{equation}
where, the first coefficients are
\[
\begin{split}
&\widetilde S_{2[(d-1)p+q]}  = (-1)^{(d-1)p + q} B_0 \Sigma_L B_0^\top, \\
&\widetilde S_{2[(d-1)p+q]-1} =
(-1)^{(d-1)p + q-1} \left( B_0 \Sigma_L S_{(d-1)p + q-1}^\top
- S_{(d-1)p + q-1} \Sigma_L B_0^\top \right).
\end{split} 
\]
From (\ref{eq:S-pol}) we see that for $j = 0,1,\ldots$
\[
\widetilde S_{2j}^\top =\widetilde S_{2j}\ \text{and } \ 
\widetilde S_{2j+1}^\top = - \widetilde S_{2j+1}.
\]
In the latter matrices the main diagonal is 0. In particular, they are 0 in
the one-dimensional case.

From (\ref{eq:R-pol}) we see that each component of the matrix $z^k R(z)$ decreases at 
least as $z^{-2}$
for $k \leq 2(p-q) -2$, thus from the definition of the residue at infinity follows
$\Theta_k = 0$. For $k = 2(p-q) -1$
it is easy to determine the limit matrix from the definition of $R(z)$,
which is $\Theta_{2(p-q)-1} = (-1)^{p-q} B_0 \Sigma_L B_0^\top$.
To determine further coefficients note that
\begin{equation} \label{eq:Theta-R}
\sum_{k=0}^\infty \Theta_k z^k = \frac{R(z^{-1})}{z}.
\end{equation}
As $\det P(z) = \prod_{j=1}^{pd} (z- \lambda_j)$ we obtain
\[
z^{-1} R(z^{-1}) = \frac{(-1)^{pd} z^{2pd - 1} }{\prod_{j=1}^{pd} (1- \lambda_j^2 z^2)}
\sum_{j=0}^{2[(d-1)p+q]} \widetilde S_j z^{-j}.
\]
This formula implies a linear recursion for the coefficients $\Theta_k$, which in a 
special case is spelled out in Section \ref{sect:example}.

Let us define the coefficients $\widetilde c_k(\omega)$ for $\omega \neq 0$,
$\omega \in (-\pi, \pi)$, via the series expansion
\begin{equation} \label{eq:def-ctilde}
\frac{1}{1 - e^{z + \i \omega} } = \sum_{k=0}^\infty \widetilde c_k(\omega) z^k,
\quad |z| < |\omega|.
\end{equation}
Let $f_\Delta$ denote the spectral density matrix of the sampled process
$Y_n^{\Delta} = Y(n\Delta)$. In the following we obtain a Taylor-expansion in $\Delta$ 
for the spectral density matrix.

\begin{proposition} \label{prop:spect1}
We have for $\omega \neq 0$, $\omega \in (-\pi, \pi)$
\[
f_\Delta (\omega) = -\frac{1}{2\pi} \sum_{k=2(p-q)-1}^\infty
(- \Delta)^k \Theta_k \widetilde c_k(\omega).
\]
\end{proposition}

From the leading term, we obtain

\begin{corollary} \label{cor:1storder}
As $\Delta \downarrow 0$
\[
f_\Delta(\omega) = (-1)^{p-q} 
\frac{\Delta^{2(p-q) -1}(1 + O(\Delta))}{2 \pi} \widetilde c_{2(p-q)-1}(\omega) B_0 
\Sigma_L B_0^\top.
\]
In particular, the first order approximation is real.
\end{corollary}

By Lemma \ref{lemma:c} below $\widetilde c_{2(p-q)-1}(\omega)$ is 
real, which implies that the first order approximation is real.

We also give another representation of the spectral density, from which the moving
average representation (\ref{eq:ma-repr}) below follows.

\begin{proposition} \label{prop:spect2}
Assume that $\lambda_1, \ldots, \lambda_{pd}$ are different zeros of $\det P(z)$,
i.e.~each has multiplicity one. Then 
\begin{equation} \label{eq:feq2}
f_\Delta(\omega)  
 = \frac{1}{4 \pi} \sum_{\ell=1}^{pd}
\frac{e^{-\i \omega} (\alpha(\ell) - \alpha(\ell)^\top) - e^{\Delta \lambda_\ell} 
\alpha(\ell) + e^{-\Delta \lambda_\ell} \alpha(\ell)^\top}
{\cosh  \Delta \lambda_\ell  - \cos \omega},
\end{equation} 
where the coefficient matrices $\alpha(\ell)  \in M_d$ come from the 
partial fraction decomposition of $R(z)$, i.e.
\begin{equation} \label{eq:pfd-s}
R(z)= 
\frac{S(z)}{\det P(z) \det P(-z)} = \sum_{\ell = 1}^{pd} 
\left( \frac{\alpha(\ell)}{z - \lambda_\ell} 
+ \frac{\beta(\ell)}{-z - \lambda_\ell} \right).
\end{equation}
\end{proposition}

The further assumption on the multiplicity of the zeros is not necessary but it makes
the formulas simpler. The following proposition gives the spectral density in the general 
case. Note that, as an abuse of notation, now $\lambda_1, \ldots, \lambda_m$ are the 
\textit{different} zeros of $\det P(z)$.

\begin{proposition} \label{prop:spect2-gen}
Let $\lambda_1, \ldots, \lambda_m$ be the different zeros of $\det P(z)$ with
multiplicity $\nu_1, \ldots, \nu_m$. Then 
\begin{equation} \label{eq:feq2-gen}
\begin{split}
f_\Delta(\omega)  
& = \frac{1}{4 \pi} \sum_{\ell=1}^{m}
\frac{e^{-\i \omega} (\alpha(\ell,1) - \alpha(\ell,1)^\top) - e^{\Delta \lambda_\ell} 
\alpha(\ell,1) + e^{-\Delta \lambda_\ell} \alpha(\ell,1)^\top}
{\cosh  \Delta \lambda_\ell  - \cos \omega} \\
& \phantom{=} \,
+ \sum_{\ell=1}^m \sum_{j=2}^{\nu_\ell} \frac{s_j^{\Delta}(\omega, \lambda_\ell)
\alpha(\ell, j)^\top + s_j^{\Delta}(-\omega, \lambda_\ell) \alpha(\ell, j)}
{(\cosh  \Delta \lambda_\ell  - \cos \omega)^j},
\end{split}
\end{equation} 
where $s_j^\Delta(\omega, \lambda_\ell)$ are trigonometric polynomials of $\omega$ of
degree $j-1$, whose coefficients depend on $\Delta$ and $\lambda_\ell$, and
the coefficient matrices $\alpha(\ell,j) \in M_d$ come from the partial
fraction decomposition of $R(z)$, i.e.
\begin{equation} \label{eq:pfd-s-gen}
R(z)= 
\frac{S(z)}{\det P(z) \det P(-z)} = \sum_{\ell = 1}^{m} \sum_{j=1}^{\nu_\ell}
\left( \frac{\alpha(\ell, j)}{(z - \lambda_\ell)^j} 
+ \frac{\beta(\ell,j)}{(-z - \lambda_\ell)^j} \right).
\end{equation}
\end{proposition}

\section{Moving average representation} \label{sect:ma-repr}

Recall that $\lambda_1, \ldots, \lambda_{pd}$ denote the zeros of $\det P(z)$. Define the 
polynomial
\[
\Phi^{\Delta}(z) = \prod_{i=1}^{pd} \left( 1 - e^{\Delta \lambda_i} z \right),
\]
and consider the filtered process
\begin{equation} \label{eq:defX}
X_n^\Delta = \Phi^\Delta(B) Y_n^\Delta,
\end{equation}
where $B$ is the backshift operator.
Note that, whenever $\lambda_i$ is complex, its complex conjugate is also a root of
$\det P(z)$, thus
the polynomial $\Phi^\Delta$ has real coefficients.
The power transfer function (see Theorem 4.4.1 in Brockwell and Davis 
\cite{BrockwellDavis})
of the filter $\Phi^\Delta$ is 
\begin{equation} \label{eq:def-phi}
\phi^\Delta(\omega) = 
\left| \prod_{i=1}^{pd} \left( 1 - e^{\Delta \lambda_i+ \i \omega}  \right) \right|^2=
2^{pd} e^{\Delta \sum_{j=1}^{pd} \lambda_j}
\prod_{j=1}^{pd} (\cosh \Delta \lambda_j - \cos \omega).
\end{equation}
When the zeros of $\det P(z)$ have multiplicity one, 
from (\ref{eq:feq2}) and (\ref{eq:def-phi}) we see that the spectral density of the 
filtered process $X^\Delta$ is
\begin{equation} \label{eq:Xd-f}
\begin{split}
f_{\textrm{\tiny{MA}}}^\Delta (\omega) & =  
\frac{2^{pd} e^{\Delta \sum_{j=1}^{pd} \lambda_j}}{4 \pi} \prod_{j=1}^{pd} (\cosh \Delta 
\lambda_j - \cos \omega)
\sum_{\ell =1}^{pd}
\frac{e^{-\i \omega} (\alpha(\ell) - \alpha(\ell)^\top) - e^{\Delta \lambda_\ell} 
\alpha(\ell) + e^{-\Delta \lambda_\ell} \alpha(\ell)^\top}
{\cosh  \Delta \lambda_\ell  - \cos \omega} \\
& = \frac{2^{pd} e^{\Delta \sum_{j=1}^{pd} \lambda_j}}{4 \pi} 
\sum_{\ell = 1}^{pd} \prod_{j \neq \ell}  (\cosh \Delta \lambda_j - \cos \omega)
\left[ e^{-\i \omega} (\alpha(\ell) - \alpha(\ell)^\top) - e^{\Delta \lambda_\ell} 
\alpha(\ell) + e^{-\Delta \lambda_\ell} \alpha(\ell)^\top \right] .
\end{split}
\end{equation}
This is clearly a trigonometric polynomial of degree less than or equal to $pd$. However, 
the coefficient of
$(\cos \omega)^{pd}$ and of $(\cos \omega)^{pd -1 } \sin \omega$ is a multiple of
\[
\sum_{\ell =1}^{pd} ( \alpha(\ell) - \alpha(\ell)^\top ). 
\]
Since in the partional fraction decomposition in (\ref{eq:pfd-s})
$\deg S(z) \leq 2((d-1) p +q)$, the sum above is necessarily 0.
Therefore $f_{\textrm{\tiny{MA}}}^\Delta (\omega)$ is a trigonometric polynomial of 
degree 
strictly less than $pd$. In the general case, when the zeros are not necessarily 
different, by (\ref{eq:feq2-gen}) we obtain a similar representation.
Thus, we have shown the following generalization of Theorem 4.2 by Schlemm and Stelzer 
\cite{SS}.

\begin{corollary} \label{corr:ma-repr}
For any $\Delta > 0$ the filtered process $X^\Delta$ is a $d$-dimensional moving average  
process of order less
than $pd$; i.e., there exist a matrix polynomial $\Psi^\Delta$ with degree less than 
$pd$ and a white noise sequence $Z^\Delta$, such that
\begin{equation} \label{eq:ma-repr}
X_n^\Delta = \Phi^\Delta(B) Y_n^\Delta = \Psi^\Delta(B) Z_n^\Delta.
\end{equation}
\end{corollary}

Schlemm and Stelzer also showed that the matrix polynomial $\Psi^\Delta$ can be chosen 
such that
$\det \Psi^\Delta(z)$ has no zeros in the closed unit disc.
Note that in Theorem 4.2 in \cite{SS} it is assumed that the zeros of $\det P(z)$
are different, while the corollary above holds in general.
In the one-dimensional case this result was shown by Brockwell and Lindner
\cite[Lemma 2.1]{BL},
see also Proposition 3 by Brockwell et al.~\cite{BDY}.
Our aim in this paper is to obtain further information on the sequence $\Psi^\Delta$ as
$\Delta \to 0$.

In order to state the asymptotic result for the moving average process we need a lemma
about the coefficients $\widetilde c_k(\omega)$.

\begin{lemma} \label{lemma:c}
There exists
polynomials $q_{k-1}, r_{k-1}$ of degree $k-1$ with real coefficients such that
\begin{equation} \label{eq:def-rq}
\begin{split}
(2k-1)! \, [2(1-\cos \omega)]^{k} \, \widetilde c_{2k-1}(\omega)
& = (-1)^k \, q_{k-1}(\cos \omega), \\ 
(2k)! \, [2(1-\cos \omega)]^{k+1} \, \widetilde c_{2k}(\omega) 
& = (-1)^k \, \i \sin \omega \cdot r_{k-1}(\cos \omega).
\end{split}
\end{equation}
Moreover,
\[
q_{k-1} (x) = (-1)^{k-1} 2^{k-1} \prod_{j=1}^{k-1} (1 - x - \xi_{2k-1,j}),
\ \text{and } \ \prod_{j=1}^{k-1} \xi_{2k-1,j} =(2k-1)! 2^{-(k-1)},
\]
and
\[
r_{k-1}(x) =
(-1)^{k-1} 2^k \prod_{j=1}^{k-1} (1 - x - \xi_{2k,j}),
\ \text{and } \ \prod_{j=1}^{k-1} \xi_{2k,j} = (2k)! 2^{-k}.
\]
For the zeros of these polynomials $\xi_{2k-1, j}, \xi_{2k, j} \not\in (0,2)$,
$j=1,2,\ldots, k-1$. 
\end{lemma}

The first few polynomials and the numerical values of the corresponding roots are
\[
\begin{split}
q_0(x) & = 1, \\
q_1(x) & = 2 (x+ 2), \{ -2\}, \\
q_2(x) & = 4 (x^2 + 13 x + 16 ), \{ -11.623, -1.377 \}, \\
q_3(x) & = 8 (x^3 + 60 x^2 + 297 x + 272), \{ -54.657, -4.141, -1.202 \}, \\
q_4(x) & = 16 (x^4 + 251 x^3 + 3651 x^2 + 10841 x + 7936 ), \{-235.705, -11.59, 
-2.579, -1.126 \}, \\
q_5(x) & = 32 (x^5 + 1018 x^4 + 38158 x^3 + 274418 x^2 + 580013 x + 353792),\\
& \phantom{=} \ \{-979.322, -30.003, -5.615, -1.973, -1.087 \},
\end{split}
\]
\[
\begin{split}
r_0(x) & = 2, \\
r_1(x) & = 4(x+ 5), \{ -5 \}, \\
r_2(x) & = 8 (x^2 + 28 x + 61),  \{ -25.619, -2.381 \}, \\
r_3(x) & = 16 (x^3 + 123 x^2 + 1011 x + 1385 ),  \{ -114.258, -7.014, -1.728 \}, \\
r_4(x) & = 32 (x^4 + 506 x^3 + 11706 x^2  + 50666 x + 50521),   
\{-481.928, -18.784, -3.832, -1.457 \}, \\
r_5(x) & = 64 (x^5 + 2041 x^4 + 118546 x^3 + 1212146 x^2 + 3448901 x +  2702765), \\
 & \phantom{=} \ \{ -1981.48, -47.391, -8.116, -2.697, -1.315 \}.
\end{split}
\]
One sees that the polynomials have real roots, moreover, the roots have the 
interlacing property. However, we cannot prove this.
Since the zeros of an orthogonal polynomial sequence have this property (see
Chihara \cite[Theorem I.5.3]{Chi}) it is tempting to think that the polynomial sequences 
above are orthogonal. However, it is easy to check that the recursion in Favard's 
Theorem (\cite[Theorem I.4.4]{Chi}) does not hold even for the first terms, 
therefore, neither of the two sequences of polynomials is orthogonal with any weight 
function.

\medskip

For $\xi \in \C$ let us define $\eta(\xi) = 1 - \xi \pm \sqrt{\xi^2 - 2 \xi}$, where the 
sign is chosen so that $| \eta(\xi)| < 1$.
Now we can state the main result of the paper.

\begin{theorem} \label{th:marepr}
The moving average process $X_n^\Delta = \Psi^\Delta(B) Z_n^\Delta$ has the asymptotic 
form 
\[
X_n^\Delta = (I_d - I_d B)^{p(d-1)+q}  \prod_{j=1}^{p-q-1} ( 1 - \eta(\xi_{2(p-q)-1,j}) 
B) Z_n, 
\
Z_n \sim \mathrm{WN}(0, \Sigma_Z), \ \textrm{as } \Delta \downarrow 0,
\]
where
\[
\Sigma_Z = \frac{\Delta^{2(p-q)-1}}{(2(p-q)-1)! \prod_{j=1}^{p-q-1} 
|\eta(\xi_{2(p-q)-1,j})|} B_0 \Sigma_L B_0^\top. 
\]
\end{theorem}

We note that $\eta(\xi_{2(p-q)-1,j})$ might be non-real (although we conjecture that they 
are all real valued),
in which case $\eta(\overline \xi_{2(p-q)-1,j})$ also appears in the product, which means 
that the coefficients
in the moving average expansion are real.

It is interesting to observe that up to the first order asymptotic the matrix moving 
average polynomial is in fact 
a scalar polynomial, and the covariance structure only appears in the covariance matrix 
$\Sigma_Z$ of the white noise. 
Thus Theorem \ref{th:marepr} has the same form as the first order version of Theorem 1 in 
\cite{BFK2};
see also Theorem 1 in \cite{BFK1}. Finally, we mention that the corresponding 
higher-order version of Theorem \ref{th:marepr},
the analog of Theorem 1 in \cite{BFK1}, seems hopeless to prove. The proof breaks down 
on 
the factorization of the 
spectral density of the moving average process, since for matrix spectral density no 
factorization holds in general;
compare Theorem 10 and 10' in Hannan \cite[Chapter II]{Hannan}.

\section{An example} \label{sect:example}

Let us consider the simplest possible non-univariate case. That is 
$p=1$, $q=0$, $d=2$. Then
\[
P(z) = I_d z + A_1, \ Q(z) = B_0, \quad A_1, B_0 \in M_2. 
\]
Moreover,
\[
\adj P(z) Q(z) = z B_0 + \adj A_1 B_0, 
\]
that is in formula (\ref{eq:S-pol}) $S_1 = B_0, S_0 = \adj A_1 B_0$. Furthermore, 
in (\ref{eq:R-pol}) we have
\begin{equation} \label{eq:R-form} 
\begin{split}
R(z) & = \frac{1}{\det P(z) \det P(-z)} ( z  S_1 + S_0 ) \Sigma_L (-z S_1^\top + 
S_0^\top) \\
& = \frac{1}{\det P(z) \det P(-z)} \left( z^2 \widetilde S_2 + z \widetilde S_1 + 
\widetilde S_0 \right),
\end{split}
\end{equation}
with
\begin{equation} \label{eq:Stilde}
\widetilde S_2 = - B_0 \Sigma_L B_0^\top, \ 
\widetilde S_1 = B_0 \Sigma_L B_0^\top (\adj A_1)^\top - \adj A_1 B_0 \Sigma_L B_0^\top, \
\widetilde S_0 = \adj A_1 B_0 \Sigma_L B_0^\top (\adj A_1)^\top.
\end{equation}
Assume that the zeros of $\det P(z)$ are different.
From (\ref{eq:R-form}) we can compute the matrices in the partional fraction 
decomposition. Standard
calculation gives that the matrices in Proposition \ref{prop:spect2} are
\begin{equation*}
\begin{split}
\alpha(1) & = \frac{1}{2 \lambda_1 (\lambda_1^2 - \lambda_2^2)}
( \widetilde S_2 \lambda_1^2 + \widetilde S_1 \lambda_1 + \widetilde S_0), \\
\alpha(2) & = \frac{-1}{2 \lambda_2 (\lambda_1^2 - \lambda_2^2)}
( \widetilde S_2 \lambda_2^2 + \widetilde S_1 \lambda_2 + \widetilde S_0) .
\end{split}
\end{equation*}
Then using formula (\ref{eq:Xd-f}), lengthy but straightforward calculation gives
\begin{equation*} \label{eq:f-specform}
\begin{split}
f_{\textrm{\tiny{MA}}}^\Delta(\omega) = \frac{2 e^{\Delta (\lambda_1+\lambda_2)}}{2 \pi 
(\lambda_1^2 - \lambda_2^2)}
\Bigg[ & \cos \omega \left(
\widetilde S_0 \left( \frac{\sinh \lambda_1 \Delta}{\lambda_1} -  \frac{\sinh \lambda_2 
\Delta}{\lambda_2} \right)
+ \widetilde S_2 (\lambda_1 \sinh \lambda_1 \Delta - \lambda_2 \sinh \lambda_2 \Delta) 
\right) \\
& \ + \i \sin \omega \cdot \widetilde S_1 (\cosh \lambda_1 \Delta - \cosh \lambda_2 
\Delta) \\
& \ + \widetilde S_0 \left( \frac{\cosh \lambda_1 \Delta \cdot \sinh \lambda_2 
\Delta}{\lambda_2} -
\frac{\cosh \lambda_2 \Delta \cdot \sinh \lambda_1 \Delta}{\lambda_1} \right) \\
& \ + \widetilde S_2 (\lambda_2 \cosh \lambda_1 \Delta \cdot \sinh \lambda_2 \Delta
- \lambda_1 \cosh \lambda_2 \Delta \cdot \sinh \lambda_1 \Delta)
\Bigg].
\end{split}
\end{equation*}
The corresponding process is MA(1), and according to Theorem 10' of Hannan \cite{Hannan} 
there is a
positive symmetric real matrix $\Psi_0$ and a real matrix $\Psi_1$, such that
\[
f_{\textrm{\tiny{MA}}}^\Delta(\omega) = \frac{1}{2 \pi} (\Psi_0 + \Psi_1 e^{\i 
\omega})(\Psi_0 + \Psi_1^\top e^{-\i \omega} ).
\]
After short calculation one sees that the first order expansion is $X_n \sim (I- B) Z_n$, 
with
covariance matrix $\Sigma_Z = \Delta B_0 \Sigma_L B_0^\top$, as we have shown in 
Theorem \ref{th:marepr}.
However, in general determining exactly the matrices $\Psi_0, \Psi_1$ is difficult.

We can also use Proposition \ref{prop:spect1}.
Combining (\ref{eq:R-form}) and (\ref{eq:Theta-R}) the $\Theta_k$ matrices can be 
calculated via the formula
\[
\sum_{k=1}^\infty \Theta_k z^k = \frac{z^3}{(1-\lambda_1^2 z^2) ( 1 - \lambda_2^2 z^2)}
\left( z^{-2} \widetilde S_2 + z^{-1} \widetilde S_1 + \widetilde S_0 \right).
\]
Multiplying by $(1-\lambda_1^2 z^2) ( 1 - \lambda_2^2 z^2)$ and equating the coefficients 
we obtain
\begin{equation} \label{eq:theta-recursion}
\begin{split}
\widetilde S_2 & =\Theta_1  \\
 \widetilde S_1 & = \Theta_2 \\
 \widetilde S_0 & = \Theta_3 - (\lambda_1^2 + \lambda_2^2) \Theta_1 \\
0 & = \Theta_4 - (\lambda_1^2 + \lambda_2^2) \Theta_2  \\
0 & = \Theta_k - (\lambda_1^2 + \lambda_2^2) \Theta_{k-2} + \lambda_1^2 \lambda_2^2 
\Theta_{k-4}, \ k \geq 5. 
\end{split}
\end{equation}
We note that also in the general case there exists a (more complicated) linear recursion 
for the $\Theta_k$ matrices.
Expanding $\cosh \Delta \lambda_i$ in a Taylor series, combining with Proposition 1,
we obtain 
\[
f_{\textrm{\tiny{MA}}}^\Delta(\omega) = - \frac{4 e^{\Delta (\lambda_1 + \lambda_2)}}
{2 \pi}\sum_{k=1}^\infty \Delta^k C_k(\omega),
\]
where $C_k(\omega)$ are trigonometric polynomials. Using the first few
values of the coefficient functions $\widetilde c_k(\omega)$ and 
(\ref{eq:theta-recursion})
\[
\begin{split}
C_1(\omega) & = (1-\cos \omega) \frac{\Theta_1}{2} \\
C_2(\omega) & = - \i \sin \omega  \frac{\Theta_2}{4} \\
C_3(\omega) & = \frac{1}{4} \left( \Theta_1 (\lambda_1^2 + \lambda_2^2 )- \Theta_3 
\frac{2 + \cos \omega}{3} \right) \\
C_4(\omega) & = - \frac{\i \sin \omega}{48} \Theta_4.
\end{split}
\]
Thus we may obtain a higher order expansion of the spectral density, e.g.
\[
\begin{split}
&f_{\textrm{\tiny{MA}}}^\Delta(\omega) = 
\frac{\Delta}{2 \pi} \bigg( -2 \widetilde S_2 (1 - \cos \omega) +
\Delta \left[ -2 (\lambda_1 + \lambda_2) \widetilde S_2 (1 - \cos \omega) + \i \sin 
\omega \, \widetilde S_1 \right] \\
& - \Delta^2 \bigg( (1 - \cos \omega ) \bigg[ \frac{\widetilde S_0 + \widetilde S_2 
(\lambda_1^2 + \lambda_2^2) }{3} 
+ \widetilde S_2 (\lambda_1 + \lambda_2)^2 \bigg] - \widetilde S_0
- \i \sin \omega \, \widetilde S_1 (\lambda_1 + \lambda_2) \bigg) + O(\Delta^3) \bigg),
\end{split}
\]
from which the statement of Theorem \ref{th:marepr} again follows. However,
it is not clear how to obtain higher order expansions for  the process itself.

\section{Proofs} \label{sect:proofs}

\begin{proof}[Proof of Proposition \ref{prop:spect1}.]
Let $\Gamma(t)$ denote the covariance matrix, i.e.~$\Gamma(t) = \E Y_0 Y_t^\top$. Then
\[
\Gamma(t) = \int_\R e^{\i t \lambda} f_Y(\lambda) \d \lambda, \quad t \in \R,
\]
and applying Cauchy's theorem componentwise we have for $t > 0$
\[
\Gamma(t) = \frac{1}{2 \pi \i} \int_\rho e^{tz} R(z) \d z,
\]
where $\rho \subset (-\infty, 0) \times \i \R$ is a closed curve, which encircles the 
zeros
of $\det P(z)$. Since $\Gamma$ is a covariance matrix, $\Gamma(-t) = \Gamma(t)^\top$.
It is clear that the autocovariance function of the discrete process
$(Y_{n\Delta})_{n \in \N}$ is $\Gamma(\Delta n)$, so by the inversion formula for 
discrete processes the spectral density can be calculated as
\begin{equation} \label{eq:f-eq1}
\begin{split}
f_\Delta(\omega) & = \frac{1}{2 \pi} \sum_{k= -\infty}^\infty e^{-\i k \omega } 
\Gamma(\Delta k) \\
& = \frac{1}{2 \pi} \left[ \sum_{k= -\infty}^0 e^{-\i k \omega } \Gamma(-\Delta k)^\top
+ \sum_{k= 1}^\infty e^{-\i k \omega } \Gamma(\Delta k) \right] \\
& = \frac{1}{4 \pi^2 \i} \left[ \int_\rho \sum_{k=0}^\infty e^{k (\Delta z + \i \omega)} 
R(z)^\top \d z
+ \int_\rho \sum_{k=1}^\infty e^{k (\Delta z - \i \omega)} R(z) \d z
\right] \\
& = \frac{1}{4 \pi^2 \i} \left[ \int_\rho \frac{1}{1 - e^{\Delta z + \i \omega}} 
R(z)^\top \d z
+ \int_\rho \frac{e^{\Delta z - \i \omega}}{1 - e^{\Delta z - \i \omega}} R(z) \d z
\right], \quad \omega \in (-\pi, \pi),
\end{split}
\end{equation}
where the change of the sum and integration is justified, since $\Re z < 0$ on the curve 
$\rho$.

Consider the Laurent-series
\[
\frac{1}{1 - e^{z + \i \omega} } = \sum_{k=-1}^\infty \widetilde c_k(\omega) z^k, \ 
\text{and } \ 
\frac{e^{z - \i \omega}}{1 - e^{z - \i \omega} } = \sum_{k=-1}^\infty \widetilde 
d_k(\omega) z^k.
\]
Note that either function has a pole at 0 only if $\omega = 0$, therefore for
$\omega \neq 0$ the series
above are usual Taylor-series. In the following we assume that $\omega \neq 0$. Adding 
the two expressions 
\[
\frac{1}{1 - e^{z + \i \omega} } + \frac{e^{z - \i \omega}}{1 - e^{z - \i \omega} }
= - \frac{e^z - e^{-z}}{e^z + e^{-z} - 2 \cos \omega} = - \frac{ \sinh z}{ \cosh z - \cos 
\omega},
\]
which is an odd function of $z$, therefore in the series expansion
\[
\widetilde c_{2k}(\omega) + \widetilde d_{2k}(\omega) = 0.
\]
On the other hand
\[
\frac{e^{z - \i \omega}}{1 - e^{z - \i \omega} } - \frac{1}{1 - e^{z + \i \omega} }  
= - \frac{e^z + e^{-z} - 2 e^{-\i \omega}}{e^z + e^{-z} - 2 \cos \omega} =
- \frac{ \cosh z - e^{-\i \omega}}{ \cosh z - \cos \omega}
= -1 - \i \frac{\sin \omega }{ \cosh z - \cos \omega}
\]
is an even function of $z$, therefore
\[ 
\widetilde d_{2k+1}(\omega) - \widetilde c_{2k+1}(\omega) = 0.
\] 
Summarizing, we obtain
\begin{equation} \label{eq:cd}
\widetilde d_k(\omega) = (-1)^{k+1} \widetilde c_k(\omega).
\end{equation}
For the coefficient $\widetilde c_k(\omega)$ we have
\begin{equation} \label{eq:c-tildec}
-\frac{1}{2} \frac{\sinh z}{\cosh z - \cos \omega} = \sum_{k=0}^\infty \widetilde 
c_{2k+1}(\omega) z^{2k+1},
\end{equation}
and
\[
\frac{1}{2}  + \frac{\i}{2} \frac{\sin \omega}{\cosh z - \cos \omega} = \sum_{k=0}^\infty 
\widetilde c_{2k}(\omega) z^{2k},
\]
so for $k \geq 1$ the coefficient $\widetilde c_{2k}(\omega)$ is purely imaginary.
From (\ref{eq:c-tildec}), using the notation of \cite{BFK1} formula (18) we see that
$\widetilde c_{2k+1} (\omega) = -c_k(\omega)/2$.

Inserting the series expansion into (\ref{eq:f-eq1}) and using that $R(z)^\top = R(-z)$ 
we obtain
\begin{equation} \label{eq:fdelta}
f_\Delta(\omega)
= \frac{1}{4 \pi^2 \i} \sum_{k=0}^\infty \Delta^k \left(
\widetilde c_k(\omega) \int_\rho  z^k R(-z) \d z + \widetilde d_k(\omega) \int_\rho z^k 
R(z) \d z \right).
\end{equation}
Then, changing the variables and using (\ref{eq:cd}) and (\ref{eq:defTheta}) we have
\begin{equation} \label{eq:c+d}
\begin{split}
& \widetilde c_k(\omega) \int_\rho  z^k R(-z) \d z + \widetilde d_k(\omega) \int_\rho z^k 
R(z) \d z \\
& = (-1)^{k+1} \widetilde c_k(\omega) \int_{- \rho}  z^k R( z) \d z + (-1)^{k+1} 
\widetilde c_k(\omega) \int_\rho z^k R(z) \d z \\
& = (-1)^{k+1} \widetilde c_k(\omega) \int_{\rho \cup - \rho} z^k R(z) \d z \\
& = (-1)^{k+1} \widetilde c_k(\omega) 2 \pi \i \Theta_k.
\end{split}
\end{equation}
Substituting into (\ref{eq:fdelta})
\[
f_\Delta (\omega) = \frac{-1}{2\pi} \sum_{k=0}^\infty  (- \Delta)^k \Theta_k \widetilde 
c_k(\omega).
\]
Taking into account that $\Theta_k = 0$ for $ k \leq 2(p-q) -2$, the proof is ready.
\end{proof}

\begin{proof}[Proof of Proposition \ref{prop:spect2}.]
We have
\[
R(z) = \frac{\adj P(z) Q(z) \Sigma_L (\adj P(-z) Q(-z))^\top}{\det P(z) \det P(-z)} 
= \frac{S(z)}{\det P(z) \det P(-z)}.
\]
Since the degree of the numerator is less than that of the denominator, for the partional 
fraction decomposition
(\ref{eq:pfd-s}) holds for some matrices $\alpha(\ell), \beta(\ell) \in M_d$.
Note that $S(-z)^\top = S(z)$  implies $\beta(\ell)^\top = \alpha (\ell)$.
By simple properties of the residue the second summand in (\ref{eq:f-eq1}) is
\[
\begin{split}
\frac{1}{2 \pi \i}  \int_\rho \frac{e^{\Delta z - \i \omega}}{1 - e^{\Delta z - \i 
\omega}} R(z) \d z
& = \sum_{\ell =1}^{pd} \frac{1}{2 \pi \i}  \int_\rho \frac{e^{\Delta z - \i \omega}}{1 - 
e^{\Delta z - \i \omega}} 
\left( \frac{\alpha(\ell)}{z - \lambda_\ell} +  \frac{\beta(\ell)}{-z - \lambda_\ell} 
\right) \d z \\
& = \sum_{\ell =1}^{pd} \frac{e^{\Delta \lambda_\ell - \i \omega}}{1 - e^{\Delta 
\lambda_\ell - \i \omega}} \alpha(\ell).
\end{split}
\]
In the same way
\[
\frac{1}{2 \pi \i}  \int_\rho \frac{R(z)^\top}{1 - e^{\Delta z + \i \omega}} \d z
= \sum_{\ell =1}^{pd} \frac{1}{1 - e^{\Delta \lambda_\ell + \i \omega}} \alpha(\ell)^\top.
\]
Therefore, by (\ref{eq:f-eq1})
\[
f_\Delta(\omega) = \frac{1}{2 \pi} \sum_{\ell =1}^{pd}
\left(  \frac{e^{\Delta \lambda_\ell - \i \omega}}{1 - e^{\Delta \lambda_\ell - \i 
\omega}} \alpha(\ell)
+ \frac{1}{1 - e^{\Delta \lambda_\ell + \i \omega}} \alpha(\ell)^\top \right),
\]
from which simple manipulation gives (\ref{eq:feq2}).
\end{proof}

\begin{proof}[Proof of Proposition \ref{prop:spect2-gen}.]
In the general case the partional fraction decomposition of $R(z)$
reads as (\ref{eq:pfd-s-gen}), with some matrices 
$\alpha(\ell, j), \beta(\ell,j) \in M_d$.
Similarly, as in the previous case we obtain
\[
\begin{split}
\frac{1}{2 \pi \i}  \int_\rho \frac{e^{\Delta z - \i \omega}}
{1 - e^{\Delta z - \i \omega}} R(z) \d z
& = \sum_{\ell =1}^{m}  \sum_{j=1}^{\nu_\ell} 
\left( D^{(j-1)} \frac{e^{\Delta z - \i \omega}}{1 - e^{\Delta z - \i \omega}} 
\right)_{\!\!z=\lambda_\ell}
\frac{\alpha(\ell,j)}{(j-1)!},
\end{split}
\]
where $D$ stands for differentiation. Noting that
\[
\frac{e^{\Delta z - \i \omega}}{1 - e^{\Delta z - \i \omega}} = -1 + 
\frac{1}{1 - e^{\Delta z - \i \omega}},
\]
one can show that for $j \geq 2$
\[
\left( D^{(j-1)} \frac{1}{1 - e^{\Delta z - \i \omega}} \right)_{\!\!z=\lambda_\ell}
= \frac{s^\Delta_j(\omega, \lambda_\ell)}{(\cosh \Delta \lambda_\ell - \cos \omega)^j},
\]
with $s_j^\Delta(\omega, \lambda_\ell)$ being a trigonometric polynomial in $\omega$ of 
degree $j-1$, whose coefficients depend on $\Delta$ and $\lambda_\ell$.
Similarly, for $j \geq 2$
\[
\left( D^{(j-1)} \frac{1}{1 - e^{\Delta z + \i \omega}} \right)_{\!\!z=\lambda_\ell}
= \frac{s^\Delta_j(-\omega, \lambda_\ell)}{(\cosh \Delta \lambda_\ell - \cos \omega)^j}.
\]
Substituting back into (\ref{eq:f-eq1}) we obtain (\ref{eq:feq2-gen}).
\end{proof}

\begin{proof}[Proof of Lemma \ref{lemma:c}.]
Recall definition (\ref{eq:def-ctilde}).
To ease the notation put $h(z) = 1/(1 - e^{z + \i \omega})$, and $y = e^{z+\i \omega}$. 
Then, for
the first few derivatives (all the derivatives are meant with respect to $z$)
\[
h'(z) = \frac{y}{(1-y)^2}, \
h''(z) = \frac{y^2 + y}{(1-y)^3}, \
h'''(z) = \frac{y^3 + 4 y^2 + y}{(1-y)^4}.
\]
In general, using induction it is easy to see that
\[
h^{(n)}(z) = \frac{y A_n(y)}{(1-y)^{n+1}}, \ n = 1,2, \ldots,
\]
where $A_n$ is a polynomial of degree $n-1$, for which the recursion
\begin{equation} \label{eq:A-recursion}
A_{n+1}(y) = (y - y^2) A_n'(y) + A_n(y) (ny + 1)
\end{equation}
holds. These are called \textit{Eulerian polynomials} (not to be confused with 
Euler-polynomials).
The coefficients are the Eulerian numbers, i.e.~$A_n(y) = A(n,n-1) y^{n-1} + A(n,n-2) 
y^{n-2} + \ldots + A(n,0)$.
The combinatorial interpretation of the
Eulerian numbers is that $A(n, k)$ is the number of permutations of $\{1,2,\ldots, n\}$ 
in which exactly $k$ elements are greater
than the previous element. From (\ref{eq:A-recursion}) we obtain
\[
A(n+1,k) = (k+1) A(n,k) + (n+1-k) A(n,k-1). 
\]
Induction gives that $A(n,n-1) = A(n,0) =1$, and
\begin{equation} \label{eq:A-symm}
 A(n,k) = A(n,n-1-k), \ k=0,1,\ldots, n-1,
\end{equation}
that is $A_n$ is a symmetric polynomial.

Since $(1 - e^{\i \omega})(1- e^{-\i \omega}) = 2 (1 - \cos \omega)$,
from (\ref{eq:def-ctilde}) we have
\[
n! \, \widetilde c_n(\omega) = h^{(n)}(0) = \frac{e^{\i \omega} A_n(e^{\i \omega})}{(1 - 
e^{\i \omega})^{n+1}}
= \frac{e^{\i \omega} A_n(e^{\i \omega}) (1 - e^{- \i \omega})^{n+1}}{ [2(1 - \cos 
\omega)]^{n+1}}.
\]
For odd indices, $n = 2k -1$, $k=1,2,\ldots$, using (\ref{eq:A-symm})
\[
\begin{split}
& A_{2k-1}(e^{\i \omega})  = A(2k-1,0) e^{(2k-2) \i \omega} +
A(2k-1,1) e^{(2k-3) \i \omega} + \ldots +  A(2k-1,1) e^{\i \omega} 
+  A(2k-1,0) \\
&=2 e^{(k-1) \i \omega} \left[ A(2k-1,0) \cos (k-1)\omega + A(2k-1,1) \cos (k-2) \omega +
\ldots + 2^{-1} A(2k-1,k-1) \right].
\end{split}
\]
The second factor is a polynomial of $\cos \omega$ of degree $k-1$. For the first factor
$e^{\i \omega} (1 - e^{-\i \omega})^2 = -2(1-\cos \omega)$, therefore
\begin{equation} \label{eq:h-poly-odd}
\begin{split}
&(2k-1)! \, \widetilde c_{2k-1}(\omega) = 
\frac{e^{\i \omega} A_{2k-1}(e^{\i \omega}) (1 - e^{- \i \omega})^{2k}}{ [2(1 - \cos 
\omega)]^{2k}} \\
& = \frac{2 (-1)^k \left[ A(2k-1,0) \cos (k-1)\omega + A(2k-1,1) \cos (k-2) \omega +
\ldots + 2^{-1} A(2k-1,k-1) \right]}{ [2(1 - \cos \omega)]^{k}}.
\end{split}
\end{equation}
For $n=2k$, $k=1,2,\ldots$, (\ref{eq:A-symm}) implies $A_{2k}(-1) =0$,
i.e.~$A_{2k}(y) = (1+y) \widetilde A_{2k-1}(y)$, where 
\[
\widetilde A_{2k-1}(y) = \widetilde A(2k-1, 0) y^{2k-2} + \widetilde A(2k-1, 1) y^{2k-3} 
+ \ldots +
\widetilde A(2k-1, 1) y + \widetilde A(2k-1,0) 
\]
is again a symmetric polynomial of degree $2k-2$. Thus, using the calculation above, and 
that
$(1 + e^{\i \omega})(1 - e^{-\i \omega}) = 2 \i \sin \omega$ we obtain
\begin{equation} \label{eq:h-poly-even}
\begin{split}
& (2k)! \, \widetilde c_{2k}(\omega) = 
\frac{e^{\i \omega} A_{2k}(e^{\i \omega}) (1 - e^{- \i \omega})^{2k+1}}{ [2(1 - \cos 
\omega)]^{2k+1}} \\
& = \frac{4 (-1)^k \i \sin \omega \left[ \! \widetilde A(2k-1,0) \cos (k-1)\omega + 
\widetilde A(2k-1,1) \cos (k-2) \omega +
\ldots + 2^{-1} \widetilde  A(2k-1,k-1) \! \right]}{ [2(1 - \cos \omega)]^{k+1}}.
\end{split}
\end{equation}
Apart from constant factors the statement is proved.

Expressing $\cos n \omega$ as a polynomial of $\cos \omega$ serves as a definition of the 
Chebishev polynomials $T_n$,
i.e.
\[
\cos n \omega = T_n (\cos \omega) = \frac{n}{2} \sum_{k=0}^{[n/2]} (-1)^k 
\frac{(n-k-1)!}{k! (n-2k)!} (2 \cos \omega)^{n-2k}.
\]
From this we see that the coefficient of $(\cos \omega )^n$ equals $2^{n-1}$. Thus the 
coefficient of
$(\cos \omega)^{k-1}$ on the right-hand side of (\ref{eq:h-poly-odd}) is $(-1)^k 
2^{k-1}$, from which we obtain
the value of the leading coefficient. After noting that $\widetilde A(2k-1,0) = A(2k,0) 
=1$, the value of the leading coefficients
follows similarly in the even case.
Finally, from (\ref{eq:A-recursion}) we see that $A_n(1) = n!$, from which the formula 
for the product of the roots
follows.

Thus we have shown that the polynomials $q_{k-1}, r_{k-1}$ defined via
\[
\begin{split}
q_{k-1}(\cos \omega) & = (-1)^k (2k-1)! [2(1 - \cos \omega)]^k \widetilde 
c_{2k-1}(\omega), \\ 
\i \sin \omega \, r_{k-1}(\cos \omega) & = (-1)^k (2k)! [2(1 - \cos \omega)]^{k+1} 
\widetilde c_{2k}(\omega),
\end{split}
\]
have the stated properties. From (\ref{eq:h-poly-odd}) we 
see that $q_{k-1}$ and $r_{k-1}$ are linear combinations
of Chebishev polynomials
\[
\begin{split}
q_{k-1}(x) & = 2 \left[ A(2k-1,0) T_{k-1}(x) + A(2k-1,1) T_{k-2}(x) + \ldots + 2^{-1} 
A(2k-1,k-1) \right], \\
r_{k-1}(x) & = 4 \left[ \widetilde A(2k-1,0) T_{k-1}(x) + \widetilde A(2k-1,1) T_{k-2}(x) 
+ \ldots + 2^{-1} \widetilde A(2k-1,k-1) \right].
\end{split}
\]

Now we turn to the statement about the roots. Indirectly assume that $q_{k-1}$ has a 
real root in $(-1,1)$, say $\cos \omega_0$. Then from (\ref{eq:h-poly-odd}) we see that
$A_{2k-1}(e^{\i \omega_0}) = 0$. This is a contradiction, since Frobenius showed in 
1910 that the roots of the Eulerian polynomials are real (for a recent work on roots 
of generalized Eulerian polynomials see Savage and Visontai \cite{SavVis}).
Similar reasoning shows that $r_{k-1}(x)$ has no real root in $(-1,1)$.
\end{proof}

\begin{proof}[Proof of Theorem \ref{th:marepr}.]
The proof relies on analyzing the spectral density 
$f_{\textrm{\tiny{MA}}}^\Delta (\omega)$ of the process $X_n^\Delta$.

As
\[
\cosh \Delta \lambda_j - \cos \omega = 1 - \cos \omega + \sum_{\ell=1}^\infty 
\frac{(\Delta \lambda_j)^{2 \ell}}{(2\ell)!},
\]
using Corollary \ref{cor:1storder} and (\ref{eq:def-phi})
the asymptotics of the spectral density of the moving average process
$\Phi^\Delta(B) Y_n^\Delta$ is
\begin{equation} \label{eq:fma-form1}
\begin{split}
f_{\textrm{\tiny{MA}}}^\Delta (\omega) & =
\frac{-1}{2\pi} 2^{pd} e^{\Delta \sum_{j=1}^{pd} \lambda_j} \prod_{j=1}^{pd} (\cosh 
\Delta \lambda_j - \cos \omega)
\sum_{k=0}^\infty  (- \Delta)^k \Theta_k \widetilde c_k(\omega) \\
& \sim \frac{\Delta^{2(p-q) -1}}{2 \pi} 2^{pd}  (1 - \cos \omega)^{pd} \widetilde 
c_{2(p-q)-1}(\omega) \Theta_{2(p-q)-1} 
\end{split}
\end{equation}
as $\Delta \downarrow 0$.
Write
\begin{equation} \label{eq:fma-form2}
f^\Delta_{\textrm{\tiny{MA}}}(\omega) \sim  
\frac{ \Delta^{2(p-q)-1}}{2\pi}
\left[ 2 (1- \cos \omega) \right]^{pd - (p-q)} 
[2(1-\cos \omega)]^{p-q} \widetilde c_{2(p-q)-1}(\omega)
\Theta_{2(p-q)-1}. 
\end{equation}
It is clear that in (\ref{eq:fma-form2}) the factor $[2(1- \cos \omega)]^{p(d-1)+q}$
corresponds to $(I_d - I_d B)^{p(d-1)+q}$ in the moving average representation.

For the factorization of the remaining terms we need that
\[
(1 - \eta e^{\i \omega} ) (1 - \eta e^{-\i \omega}) = 
2 \eta \left( 1 - \cos \omega + \frac{(1-\eta)^2}{2\eta} \right),
\]
thus solving the equation $- \xi = (1-\eta)^2/(2\eta)$ we have for the
solution
\[
\eta(\xi) := 1 - \xi \pm \sqrt{\xi^2- 2 \xi},
\]
where the sign is chosen so that $| \eta(\xi)| < 1$.
This is possible, since the product of the two roots is 1.
That is
\[
1 - \cos \omega - \xi = \frac{1}{2 \eta(\xi)} (1 - \eta(\xi) e^{\i \omega} )
(1 - \eta(\xi) e^{-\i \omega}).
\]
Therefore combining the above with Lemma \ref{lemma:c} we obtain
\[ 
\begin{split}
[2(1-\cos \omega)]^{p-q} \widetilde c_{2(p-q)-1}(\omega) & =
- 2^{p-q-1} \frac{\prod_{j=1}^{p-q-1} (1 - \cos \omega - \xi_{2(p-q)-1,j})}
{(2(p-q)-1)!} \\
& = - \frac{\prod_{j=1}^{p-q-1} (1 - \eta(\xi_{2(p-q)-1,j}) e^{\i \omega} )
(1 - \eta(\xi_{2(p-q)-1,j}) e^{-\i \omega} )}
{(2(p-q)-1)! \prod_{j=1}^{p-q-1} \eta(\xi_{2(p-q)-1,j})}.
\end{split}
\] 

We conjecture that the zeros $\xi_{2k-1,j}$ are all real and greater than 2.
This is true for $k=1,2,\ldots, 8$, however we cannot prove it in general.
For real zeros the $\eta$'s are also real (we did prove that
$\xi_{2k-1,j} \not\in (0,2)$), thus
in the factorization everything is real. When there is a non-real root $\xi$
then necessarily its conjugate $\overline \xi$ is also a root, and one can check easily 
that $\eta(\overline \xi) = \overline{\eta(\xi)}$, therefore in the factorization
the coefficients are real.
\end{proof}

\medskip
\noindent \textbf{Acknowledgement.}
I am grateful to Claudia Kl\"uppelberg and to Peter Brockwell for inspiring
conversations on the subject and for comments on the manuscript. I also thank B\'ela 
Nagy for discussions on the polynomials appearing in Lemma \ref{lemma:c}.

\bibliographystyle{abbrv}
\bibliography{MCMA}

\end{document}